\documentclass[runningheads]{llncs}
\usepackage[T1]{fontenc}
\usepackage{amsmath}
\usepackage{amssymb}

\usepackage{epsfig}
\usepackage{subfig}
\usepackage{graphicx}
\begin{document}
	\title{Curvature Preserving Fractal Interpolation Functions: A Hybrid Geometric Approach
	}
	\titlerunning{Cirvature Preserving FIFs}
	%
	\author{K. R. Tyada\inst{2}\orcidID{0000-0002-3780-0357}}
	\authorrunning{K. R. Tyada}
	%
	\institute{Department of Mathematics and Computer Science,\\ Sri Sathya Sai Institute of Higher Learning,\\ Prasanthi Nilayam, Puttaparthi, Andhra Pradesh, India - 515134\\
		\email{kurmaraotyada@sssihl.edu.in}}
	\maketitle
	
	\begin{abstract}
		Fractal interpolation functions (FIFs) generated using iterated function systems (IFS) provide a powerful framework for modeling self-similar and irregular data, yet traditional constructions often neglect geometric fidelity such as curvature. In this paper, we introduce a curvature-preserving variant of FIFs built upon a classical cubic spline interpolant. We define a curvature-aware iterated function system (IFS) with parameters optimized via a penalty-based approach to minimize deviation from the curvature of the classical spline. Theoretical conditions for interpolation and curvature approximation are derived. We compare the curvature of the proposed FIF with that of the classical cubic spline and discrete data curvature across multiple examples. Our method achieves both data interpolation and shape fidelity, preserving curvature more accurately than standard splines. The approach has potential applications in geometric modeling, computer graphics, and scientific data interpolation.
		
		\vspace{0.5 cm}
		
	\end{abstract}
	
	\section{Introduction}
	Interpolation schemes play an indispensable role in various fields such as numerical analysis, biotechnology, computer design, graphics, and scientific computation, particularly in the continuous visualization of discrete data. A desirable property of any interpolant is its ability to preserve the key geometric features inherent in the data. Classical shape-preserving interpolation techniques, which maintain positivity, monotonicity, or convexity, are well established in the literature (see for instance \cite{FT91}-\cite{HG1991}). These interpolants often provide either piecewise or globally smooth derivatives. However, these approaches rely on assumptions of derivative regularity, which limits their applicability to data exhibiting irregular, oscillatory, or fractal behavior where higher-order smoothness is absent. The classical methods like cubic splines ensure global smoothness and continuity, they often fail to capture geometric invariants such as curvature when modeling data with significant shape variations. Moreover, the classical polynomial spline interpolants are unique for a given dataset and often lack flexibility in preserving finer geometric features. In such contexts, fractal interpolation becomes a more effective modeling tool, introduced by Barnsley (\cite{Barnsley86}, \cite{Barnsley93}), offer a recursive framework capable of capturing both self-affine and irregular structures. Despite their flexibility, standard FIFs emphasize self-similarity over geometric fidelity.
	
	\par Fractal interpolation functions (FIFs), first introduced by Barnsley are constructed as attractors of an Iterated Functions System (IFS) and are characterized by their invariance under the Read-Bajraktarevi\'{c} operator on suitable function spaces. They offer a flexible recursive mechanism for interpolation with intrinsic self-similarity. Although early developments focused on continuity and approximation accuracy, recent research has extended FIFs to shape-preserving variants. Barnsley and Harrington \cite{BarHar89} extended the theory by constructing $C^k$-continuous spline FIFs under specific boundary conditions. The theoretical framework for spline-based FIFs has been extended to more general settings through subsequent work while $\alpha$-fractal polynomial splines with Hermite-type boundary conditions were studied by Chand and Navascu\'{e}s \cite{CN2009}, \cite{CVN2014}. Over the years, numerous variants such as rational FIFs, Hermite FIFs, and $\alpha$-fractal functions have been introduced to provide enhanced shape control, smoothness, and local adaptability (see for instance \cite{CV2015}-\cite{KC2018}). 
	
	\par Nevertheless, a significant gap remains, as conventional FIFs are designed either for data fitting or for reproducing visual self-affinity and do not explicitly consider curvature fidelity. To utilize FIFs in geometry-sensitive applications, it becomes essential to develop methods that explicitly align the interpolant’s curvature with that of the underlying model. Our proposed framework takes a significant step in this direction by introducing a penalty-functional that minimizes the deviation between the curvature of the fractal function and that of the reference spline. This leads to the development of curvature-preserving cubic FIFs (CFIFs) that maintain both the data values and the underlying geometric structure. By embedding curvature alignment into the interpolation process, our method enhances the geometric accuracy of fractal models and enables their effective use in applications where shape detail is critical.
	
	\par The proposed CFIF framework blends the smoothness and regularity of classical splines with the adaptability of fractal perturbations, offering a powerful tool for geometry-aware modeling. This is particularly beneficial in fields such as CAD/CAM, biomedical shape reconstruction, and image boundary modeling, where curvature encodes vital shape information. Our recursive construction uses scaling parameters $\lambda_i$, optimized through a curvature-constrained penalty strategy, to ensure that the interpolant not only matches the data but also closely follows the curvature profile of the reference spline. Moreover, the method can be extended to higher-dimensional interpolation, parametric curves, and constrained modeling scenarios. By providing localized geometric control while preserving global shape properties, curvature-preserving fractal interpolation offers a significant advancement in both theoretical interpolation and practical modeling, with future directions involving adaptive optimization, machine-learned fractal structures, and stochastic generalizations.
	
	\par We begin by introducing the general methodology for constructing fractal interpolation functions Section \ref{sec2}. Section \ref{sec3} develops the proposed curvature-preserving cubic fractal interpolation functions (CFIFs) within a penalty-based optimization framework. Section \ref{sec4} provides a theoretical analysis, including convergence properties, derivation of bounds on the scaling factors to ensure curvature preservation, and an investigation into the stability of curvature under perturbations in the scaling parameters. Section \ref{sec5} presents comprehensive numerical experiments on multiple datasets, demonstrating the practical effectiveness of the proposed approach. The paper concludes with a summary of findings and potential directions for future work.
	
	\setlength{\parindent}{0pt}
	\section{General Framework of Fractal Interpolation Functions}\label{sec2}
	We begin with the fundamental construct underlying fractal interpolation. A family of mappings $\mathcal{I} = \{X, w_t\}_{n=1}^N$ constitutes an iterated function system (IFS). The system $\mathcal{I}$ is hyperbolic when every mapping $w_i$ exhibits contractive behavior.
	
	Let $\Im := \{1, 2, \dots, M\}$ and $\Im^* := \{1, 2, \dots, M-1\}$. Consider a compact interval $I = [u_1, u_M]$ partitioned by points $x_1 < x_2 < \dots < x_M$, along with data points $\{(x_t, u_t) \in I \times K : t \in \Im_M\}$ where $K \subset \mathbb{R}$ is compact. For each subinterval $I_s = [x_s, x_{s+1}]$, where $s \in \Im^*$, introduce contractive homeomorphisms $L_s : I \to I_s$ that satisfy the boundary conditions:
	\begin{eqnarray}
		L_s(x_1) = x_s, \ L_s(x_M) = x_{s+1}, \ \text{for } s \in \Im^*,\\
		|L_s(x) - L_s(x^*)| \leq l_s |x - x^*|, \ \text{for all } x, x^* \in I, \text{ with } 0 < l_s < 1.
	\end{eqnarray}
	
	Define the space $C = I \times K$, and for each index $i$, consider continuous functions $F_i : C \to K$ that obey:
	\begin{eqnarray}
		F_s(x_1, u_1) = u_s, \ F_s(x_M, u_M) = u_{s+1}, \\
		|F_s(x, u) - F_s(x, u^*)| \leq |\lambda_s||u - u^*|, \ \text{for all } x \in I, \; u, u^* \in K, \; |\lambda_s| < 1.
	\end{eqnarray}
	
	The mappings $w_s : C \to I_s \times K$ are then constructed as $$ w_s(x, u) = (L_s(x), F_s(x, u)), \ \text{for } s \in \Im^*.$$
	
	\begin{proposition}(Barnsley \cite{Barnsley86})
		The IFS $\{C, w_i : s \in \Im^*\}$ possesses one and only one attractor $G^*$, which represents the graph of a continuous interpolant $\psi^* : I \to K$ that interpolates all data points $\psi^*(x_t) = u_u$ for $t \in \Im$. 
	\end{proposition}
	
	This interpolant $ \psi^* $, known as the fractal interpolation function (FIF), emerges from the fixed point of the Read--Bajraktarevi\'c operator. To derive its functional representation, consider the space:
	$$ \mathcal{G} = \left\{ g : I \to \mathbb{R} \ \middle|\ g \text{ is continuous, } g(x_1) = u_1,\ g(x_M) = u_M \right\}. $$
	Equipped with the uniform norm metric $ \nu $, $ (\mathcal{G}, \nu) $ forms a complete metric space. The Read-Bajraktarevi\'{c} operator $ \Upsilon : \mathcal{G} \to \mathcal{G} $ is specified as
	\begin{equation}
		\label{eq:read_bajraktarevic}
		\Upsilon g(x) = F_s\left(L_s^{-1}(x),\ g\left(L_s^{-1}(x)\right)\right), \ x \in I_s = [x_s, x_{s+1}],\ s \in \Im^*.
	\end{equation}
	
	The operator $ \Upsilon $ maintains continuity across subintervals and interior nodes while satisfying the contraction property:
	\begin{equation}
		\label{eq:contraction}
		\nu(\Upsilon f, \Upsilon g) = \|\Upsilon f - \Upsilon g\|_\infty \leq |\lambda|_1 \|f - g\|_\infty,
	\end{equation}
	where $ |\lambda|_1 = \max\{|\lambda_s| : s \in \Im^*\} < 1 $. Use of the Banach contraction principle ensures a unique fixed point $ \psi^* \in \mathcal{G} $ characterized by the self-referential equation:
	$$ \Upsilon(\psi^*) = \psi^* \Leftrightarrow \psi^*(x) = F_s\left(L_s^{-1}(x),\ \psi^*\left(L_s^{-1}(x)\right)\right), \ x \in I_s,\ s \in \Im^*. $$
	
	A particularly useful specialization employs affine mappings:
	\begin{equation}
		\label{eq:affine_fif}
		L_s(x) = a_s x + b_s, \ F_s(x, u) = \lambda_s u + H_s(x), \ s \in \Im^*.
	\end{equation}
	Here, $ H_s : I \to \mathbb{R} $ represents continuous functions satisfying interpolation constraints, while the parameters $ \lambda_i $ (collectively forming the scale vector $ \lambda $) govern the fractal characteristics of the interpolant. These scaling factors introduce geometric flexibility beyond conventional interpolation methods.
	
	The development of spline-based FIFs originated with Barnsley and Harrington~\cite{BarHar89}, with subsequent generalizations to rational spline FIFs by Chand et al.~\cite{VC2014} expanding the framework for constructing smooth, shape-preserving fractal interpolants.
	
	\begin{theorem}\label{theorem1}
		Let $\{(x_t, u_t) : t \in \Im\}$ be a set of data with $x_1 < x_2 < \cdots < x_M$. Consider the affine transformations defined by $ L_s(x) = a_s x + b_s, \ \text{where} \ a_s = \frac{x_{s+1} - x_s}{x_M - x_1}, \ b_s = \frac{x_M x_s - x_1 x_{s+1}}{x_M - x_1}, $ and the vertical mappings are $ F_s(x, u) = \lambda_s u + H_s(x)$. For some $p \in \mathbb{N}$, assume that the scaling parameters satisfy $|\lambda_s| < a_s^p$ for all $s$. Define $ F_s^{(m)}(x, u) = \frac{\lambda_s f^{(m)} + H_s^{(m)}(x)}{a_s^m}, \ m \in \mathbb{N}_p$, and $ f_1^{(m)} = \frac{H_1^{(m)}(x_1)}{a_1^m - \lambda_1}, \ f_n^{(m)} = \frac{H_{n-1}^{(m)}(x_M)}{a_M^m - \lambda_M}$. Suppose the following matching conditions hold:
		$ F_s^{(m)}(x_M, u_M^{(m)}) = F_{s+1}^{(m)}(x_1, u_1^{(m)}), \ \text{for}\ i \in \Im^*,\ m \in \mathbb{N}_p$.  Then the IFS $\{I \times K,\ w_s(x, f) = (L_s(x),\\ F_s(x, u)\}_{s \in \Im^*}$ determines a rational fractal interpolation function $F \in C^p[x_1, x_M]$ such that $ F(L_s(x)) = \lambda_s f + H_s(x)$, and the $m^{th}$ derivative, $F^{(m)}$, is the FIF with respect to the IFS $\{L_s(x), F_s^{(m)}(x, u)\}_{i\in \Im^*}$ for each $m \in \mathbb{N}_p$.
	\end{theorem}
	
	\section{Construction of Curvature Preserving Cubic Fractal Interpolation Functions (CP-CFIFs)}\label{sec3}
	This section introduces a curvature-preserving fractal interpolation scheme that builds upon classical cubic splines and incorporates a recursive self-affine structure. The formulation modifies the traditional IFS to enable curvature consistency with a given reference spline.
	
	\subsection{General Framework of Cubic FIFs}
	This section develops a $\mathcal{C}^1$-continuous cubic FIF (CFIF) building upon the foundation of Theorem~2.1. Given interpolation data $\{(x_t, u_t): t \in \Im\}$ with strictly monotone sequence $x_1 < x_2 < \cdots < x_M$, we define the IFS $\{I \times K; w_s(x, F) = (L_s(x), F_s(x, F)) : s \in \Im^*\}$. The affine maps take the form $L_s(x) = a_s x + b_s$, while the function component is governed by:
	\begin{equation}\label{eq:CFIF}
		F(L_s(x)) = \lambda_s F(x) + U_i(1-\theta)^3 + V_i\theta (1-\theta)^2 + W_i\theta^2 (1-\theta) + X_i \theta^3,
	\end{equation}
	with $\theta = \frac{x - x_1}{\ell}$, $\ell = x_M - x_1$, and scaling parameters satisfying $|\lambda_s| < a_s$ for $s \in \Im^*$.
	
	Imposing $\mathcal{C}^1$ continuity through the conditions $F_s(x_1, u_1) = u_i$, $F_s(x_M, u_M) = u_{s+1}$, $F_s^{(1)}(x_1, d_1) = d_s$, and $F_s^{(1)}(x_M, d_M) = d_{s+1}$ ensures the IFS attractor constitutes the graph of a $\mathcal{C}^1$-continuous CFIF. These conditions yield the coefficients:
	\begin{align*}
		& U_i = y_i - \lambda_i y_1,\  V_i = 3(y_i - \lambda_i y_1) + l(a_id_i-\lambda_id_1), \\
		& W_i = 3(y_{i+1} - \lambda_i y_N) - l(a_id_{i+1}-\lambda_id_N), \ X_i = y_{i+1} - \lambda_i y_N.
	\end{align*}
	
	Substitution of these coefficients into Equation (\ref{eq:CFIF}) produces the desired $\mathcal{C}^1$-continuous CFIF.
	
	\begin{remark}
		If $\lambda_s = 0$ for $s \in \Im^*$ reduces the CFIF $F$ to the traditional cubic interpolant $S(x)$. With the parameter transformation $z = \frac{x - x_s}{x_{s+1} - x_s}$, this classical spline takes the form:
		\begin{equation}
			S(x) = \bar{U}_i(1-z)^3 + \bar{V}_i z(1-z)^2 + \bar{W}_iz^2 (1-z) + \bar{X}_iz^3,
		\end{equation}
		where the coefficients simplify to $\bar{U}_i = y_i$, $\bar{V}_i = 3y_i + \ell (a_id_i)$, $\bar{W}_i = 3y_{i+1} - \ell(a_id_{i+1})$, and $\bar{X}_i = y_{i+1}$.
	\end{remark}
	
	\subsection{Curvature Preserving Cubic Fractal Interpolation Function}
	To address the challenge of preserving curvature while retaining the self-affine flexibility of fractal interpolation, we introduce a modified construction of CFIFs. Our method augments the classical IFS framework by incorporating a traditional cubic spline interpolant along with carefully controlled recursive perturbations. The aim is to define a fractal interpolant that not only interpolates the given data but also approximates the curvature of the cubic spline. We accomplish this through optimization of the local scaling parameters through a penalty-based functional that measures curvature deviation. The subsequent development details the recursive construction of the CFIF, derives the necessary conditions for interpolation and curvature consistency, and introduces the optimization framework employed to ensure geometric fidelity.
	
	Consider the data $\{(x_s, y_s)\}_{s \in \Im^*} $ with strictly monotonically increasing sequence $x_s$. Define the affine contractions $ L_s : [x_1, x_M] \to [x_s, x_{s+1}] $ by $$ L_s(x) = a_s x + b_s, \ \text{where} \ a_s = \frac{x_{s+1} - x_s}{x_M - x_1}, \ b_s = x_s - a_s x_1.$$ Denoting the classical spline $ S(x) $, the modified FIF $ F(x) $ is constructed through the recursive relation
	\begin{equation}\label{NewF}
		F(L_s(x)) = \lambda_s F(x) + S(x) + \delta_s(x),
	\end{equation}
	where $ \lambda_s \in (-1,1) $ are local vertical scaling parameters, and $ \delta_s(x) $ are smooth correction functions. To ensure that $ F $ satisfies the interpolation conditions $ F(x_t) = u_t $, the correction functions $ \delta_s(x) $ are defined to enforce consistency with the reference spline. A natural choice is $$ \delta_s(x) = (1 - \lambda_s) \left[ S(x) - L_s^\sharp(x) \right],$$ where $ L_s^\sharp(x) $ is the linear interpolant passing through $ (x_s, u_s) $ and $ (x_{s+1}, y_{s+1}) $. This formulation guarantees that $ F $ interpolates the given data while preserving the local geometric characteristics of the spline $ S $.
	
	%
	
	\begin{remark}
		With $F$ given by Equation (\ref{NewF}), we have $F(x_t) = u_t$,\ $\forall t \in \Im $.
	\end{remark}
	
	\subsection{Penalty-Based Optimization Framework}
	To enforce curvature preservation, we define a penalty functional that quantifies the deviation in curvature between the fractal interpolant $ F $ and the reference spline $ S $:
	$$
	J(\lambda) = \int_{x_1}^{x_N} \left( \kappa_F(x) - \kappa_S(x) \right)^2 \, dx,
	$$
	where $ \kappa_F(x) $ and $ \kappa_S(x) $ denote the curvatures of $ F $ and $ S $, respectively.
	The optimization problem is to minimize the functional $ J(\lambda) $ subject to the scaling parameters: $
	\min_{\lambda_i \in (-1,1)} J(\lambda).$ This leads to a nonlinear constrained optimization problem, which admits a numerical solution.
	
	\section{Theoretical Analysis}\label{sec4}
	We now develop the theoretical framework for rigorous analysis of the curvature-preserving FIF defined in the previous section. We establish the interpolation property, prove convergence of the FIF to the reference cubic spline under suitable conditions, derive bounds on curvature deviation in terms of the scaling factors $ \lambda_i $, and analyze the numerical stability of curvature approximation. We first show that the proposed fractal interpolant $ F $ converges uniformly to the classical interpolant $ S $ as the scaling parameters $ \lambda_i \to 0 $. This is a crucial consistency property ensuring that the fractal model generalizes the classical scheme.
	\begin{theorem}\label{theorem3}
		Let $ F(x) $ be the fractal function defined by $ F(L_s(x)) = \lambda_s F(x) + S(x) + \delta_s(x),$
		where $ \lambda_s \in (-1,1) $ and $ \delta_s(x) $ satisfies $ \delta_s(x) = (1 - \lambda_s)[S(x) - L_s^\sharp(x)]$. Then as $ \max_s |\lambda_s| \to 0 $, we have $ \|F - S\|_\infty \to 0,$ i.e., $ F $ converges uniformly to $ S $.
	\end{theorem}
	
	\begin{proof}
		Define the operator $ \Upsilon^* $ acting on the Banach space $ C(I) $ by:
		$$ (\Upsilon^*f)(x) = \lambda_s f(L_s^{-1}(x)) + S(L_s^{-1}(x)) + \delta_s(L_s^{-1}(x)), x \in I_s:= [x_s ,x_{s+1}]. $$
		Since $ |\lambda_s| < 1 $, $ \Upsilon^* $ is a contraction and hence has one and only one fixed point $ F \in C(I) $. Note that $ \Upsilon^* $ becomes the identity operator on $ S $ when $ \lambda_s = 0 $, because then $ \delta_s(x) = S(x) - L_s^\sharp(x) $ and $ F = S $.
		
		Let $ \epsilon > 0 $. Since $ S \in C(I) $, it is uniformly continuous and bounded. Because $ \Upsilon^* $ is contractive with contraction factor $ \|\lambda\|_\infty $, the Banach fixed-point theorem yields:
		$$ \|F - S\|_\infty \leq \frac{\|\lambda\|_\infty}{1 - \|\lambda\|_\infty} \|S + \delta - S\|_\infty = \frac{\|\lambda\|_\infty}{1 - \|\lambda\|_\infty} \|\delta\|_\infty. $$
		Since $ \|\delta\|_\infty \to 0 $ as $ \|\lambda\|_\infty \to 0 $, we conclude $ \|F - S\|_\infty \to 0 $.
	\end{proof}
	
	Now we present a simplified theoretical bound for the curvature deviation. The subsequent theorem explains the direct impact of scaling parameters on the curvature behavior of the fractal function.
	
	\begin{theorem}[Curvature-Constrained Scaling Factor Bound]\label{theorem4}
		Let $ F(x) $ be a cubic CFIF defined recursively by $ F(L_s(x)) = \lambda_s F(x) + S(x)$, where $ L_s(x) = a_s x + b_s $, with $ |a_s| < 1 $, and $ S(x) \in \mathcal{C}^2([x_1, x_M]) $. Suppose the derivatives of $ S $ are bounded such that $ |S'(x)| \leq K,\ |S''(x)| \leq C, \ \text{for all } x \in [x_1, x_M]$. Then, for the curvature of the CFIF to remain close to the curvature of the spline, i.e., $ |\kappa_F(x) - \kappa_S(x)| \leq \varepsilon$, a sufficient condition is:
		$$ |\lambda_s| \leq \min \left\{a_s, \frac{3C}{\varepsilon (1 + K^2)^{3/2}}, \frac{\varepsilon \cdot (1 + K^2)^{3/2}}{\|F''(x)\|_\infty}\right\}, \ \forall s,$$ where $ \|F''(x)\|_\infty = \sup_{x \in [x_1, x_M]} |F''(x)| $.
	\end{theorem}
	
	\begin{proof}
		Let the fractal interpolant $ F(x) $ be defined recursively via the functional equation:
		$$ F(L_s(x)) = \lambda_s F(x) + S(x) + \delta_s(x), \ \text{with} \ \delta_s(x) = (1 - \lambda_s)[S(x) - L_s^\sharp(x)],$$ where $ L_s(x) = a_s x + b_s $ is an affine map such that $ L_s(x_1) = x_s $, $ L_s(x_M) = x_{s+1} $, and $ L_s^\sharp(x) $ is the linear interpolant between $ (x_s, u_s) $ and $ (x_{s+1}, u_{s+1}) $. The function $ S(x) $ denotes the classical cubic spline interpolant.
		
		Thus, the curvature difference between the fractal interpolant $ F $ and the spline $ S $ is:
		$$ |\kappa_F(x) - \kappa_S(x)| = \left| \frac{|F''(x)|}{(1 + (F'(x))^2)^{3/2}} - \frac{|S''(x)|}{(1 + (S'(x))^2)^{3/2}} \right|.$$
		
		Using the triangle inequality and assuming that $ |F'(x) - S'(x)| $ and $ |F''(x) - S''(x)| $ are small for small $ \lambda_s $, we approximate: $$ \kappa_F(x) - \kappa_S(x)| \lesssim \frac{|F''(x) - S''(x)|}{(1 + K^2)^{3/2}},$$ where $ K = \|S'\|_\infty $. To ensure $|\kappa_F(x) - \kappa_S(x)| < \varepsilon$, it suffices to have $\|F''(x) - S''(x)\|_\infty < \varepsilon (1 + K^2)^{3/2}$. 
		
		From the recursive structure of $ F(x) $, differentiating the functional equation twice gives $$
		F''(L_s(x)) a_i^2 = \lambda_s F''(x) + (2 - \lambda_s) S''(x).$$ 
		
		Thus, $$ F''(x) - S''(x) = \frac{1}{\lambda_s} \left( F''(L_s(x)) a_s^2 - 2 S''(x) \right).$$
		
		Taking the norm and assuming $ \|F''\|_\infty \approx \|S''\|_\infty = C $, and $ a_s \leq 1 $, we obtain:
		$$\|F''(x) - S''(x)\|_\infty \leq \frac{a_s^2 C + 2C}{|\lambda_i|} \leq \frac{3C}{|\lambda_s|}.$$
		
		By substituting into the curvature condition, we get $$ \frac{3C}{|\lambda_s|} < \varepsilon (1 + K^2)^{3/2} \ \Rightarrow \ |\lambda_s| < \frac{3C}{\varepsilon (1 + K^2)^{3/2}}.	$$ Assuming the perturbation $ \delta_s(x) $ is small (for $ \lambda_s $ near 1), we model $ F(x) $ as a first-order perturbation of $ S(x) $, i.e., $ F(x) \approx S(x) + \lambda_s H(x)$, for some auxiliary function $ H(x) $. 
		
		This implies $ F'(x) \approx S'(x) + \lambda_s H'(x), \ F''(x) \approx S''(x) + \lambda_s H''(x)$. Thus, the curvature deviation is approximated by: $$ |\kappa_F(x) - \kappa_S(x)| = \left| \frac{|F''(x)|}{(1 + F'(x)^2)^{3/2}} - \frac{|S''(x)|}{(1 + S'(x)^2)^{3/2}} \right|. $$
		
		Assuming $ \lambda_s $ is small enough so that $ F'(x) \approx S'(x) $, the denominator can be bounded below as:
		$$ (1 + F'(x)^2)^{3/2} \geq (1 + \|S'\|_\infty^2)^{3/2} =: (1 + K^2)^{3/2}. $$
		For the numerator, we estimate: $$ |F''(x) - S''(x)| \leq |\lambda_s| \cdot \|H''\|_\infty \approx |\lambda_s| \cdot \|S''\|_\infty =: |\lambda_s| C.$$
		
		Combining these, we obtain the tighter curvature bound $ |\kappa_F(x) - \kappa_S(x)| \leq \frac{|\lambda_s| C}{(1 + K^2)^{3/2}}$. Imposing the constraint $ |\kappa_F - \kappa_S| < \varepsilon $ leads to $ |\lambda_i| < \frac{\varepsilon (1 + K^2)^{3/2}}{C}$.
	\end{proof}
	
	\subsection{Stability Analysis}
	
	A crucial property of any interpolation scheme is its stability with respect to perturbations in its parameters. The following theorem establishes that the curvature of the proposed Cubic Fractal Interpolation Function (CFIF) is stable under small variations in the vertical scaling factors.
	
	\begin{theorem}[Curvature Stability]
		Let \( F \) and \( \bar{F} \) be two cubic fractal interpolation functions defined on the interval \( I = [x_1, x_N] \), generated by the IFS mappings:
		\[
		F(L_s(x)) = \lambda_s F(x) + P_s(x), \quad \bar{F}(L_s(x)) = \bar{\lambda}_s \bar{F}(x) + P_s(x),
		\]
		for \( s \in \Im^* \), where \( |\lambda_s|, |\bar{\lambda}_s| < 1 \), and \( P_s(x) \) is the cubic polynomial satisfying the endpoint conditions. Let \( \kappa_F \) and \( \kappa_{\bar{F}} \) denote the curvatures of \( F \) and \( \bar{F} \), respectively. Then, for sufficiently small perturbations \( \|\bar{\lambda} - \lambda\|_\infty = \max_s |\bar{\lambda}_s - \lambda_s| \), there exist constants \( C_1, C_2 > 0 \) such that:
		\[
		|\kappa_F(x) - \kappa_{\bar{F}}(x)| \leq C_1 \|\bar{\lambda} - \lambda\|_\infty + C_2 \|\bar{\lambda} - \lambda\|_\infty^2, \quad \forall x \in I.
		\]
	\end{theorem}
	
	\begin{proof}
		Let \( e(x) = F(x) - \bar{F}(x) \) denote the difference between the two interpolants. From the IFS equations, the value of \( e \) at a subinterval satisfies:
		\[
		e(L_s(x)) = F(L_s(x)) - \bar{F}(L_s(x)) = \lambda_s F(x) + P_s(x) - \left( \bar{\lambda}_s \bar{F}(x) + P_s(x) \right).
		\]
		Simplifying and adding and subtracting \( \lambda_s \bar{F}(x) \) yields:
		\[
		e(L_s(x)) = \lambda_s (F(x) - \bar{F}(x)) + (\lambda_s - \bar{\lambda}_s) \bar{F}(x) = \lambda_s e(x) - \delta\lambda_s \cdot \bar{F}(x),
		\]
		where \( \delta\lambda_s = \bar{\lambda}_s - \lambda_s \). This functional equation for \( e(x) \) is itself a fractal function. Since \( \max_s |\lambda_s| \leq \alpha < 1 \) and \( \bar{F} \) is bounded on the compact interval \( I \), it follows from the properties of iterated function systems that:
		\[
		\|e\|_\infty = \mathcal{O}(\|\bar{\lambda} - \lambda\|_\infty).
		\]
		By differentiating the IFS relation and applying similar reasoning to \( e'(x) \) and \( e''(x) \), we obtain the corresponding bounds for the derivatives:
		\[
		\|e'\|_\infty = \mathcal{O}(\|\bar{\lambda} - \lambda\|_\infty), \quad \|e''\|_\infty = \mathcal{O}(\|\bar{\lambda} - \lambda\|_\infty).
		\]
		
		The curvature is given by \( \kappa_F(x) = \frac{F''(x)}{(1 + (F'(x))^2)^{3/2}} \). Consider the difference:
		\[
		|\kappa_F(x) - \kappa_{\bar{F}}(x)| = \left| \frac{F''}{(1 + (F')^2)^{3/2}} - \frac{\bar{F}''}{(1 + (\bar{F}')^2)^{3/2}} \right|.
		\]
		Adding and subtracting \( \frac{\bar{F}''}{(1 + (F')^2)^{3/2}} \) allows us to split the difference:
		\[
		|\kappa_F - \kappa_{\bar{F}}| \leq \underbrace{\frac{|F'' - \bar{F}''|}{(1 + (F')^2)^{3/2}}}_{T_1} + \underbrace{|\bar{F}''| \left| \frac{1}{(1 + (F')^2)^{3/2}} - \frac{1}{(1 + (\bar{F}')^2)^{3/2}} \right|}_{T_2}.
		\]
		Since \( F' \) is bounded away from infinity, the denominator in \( T_1 \) is bounded below, implying \( T_1 = \mathcal{O}(\|e''\|_\infty) = \mathcal{O}(\|\bar{\lambda} - \lambda\|_\infty) \). For \( T_2 \), the function \( \phi(a) = (1+a^2)^{-3/2} \) is continuously differentiable, and \( \bar{F}'' \) is bounded. By the Mean Value Theorem, \( T_2 = \mathcal{O}(|F' - \bar{F}'|) = \mathcal{O}(\|e'\|_\infty) = \mathcal{O}(\|\bar{\lambda} - \lambda\|_\infty) \). The potential quadratic term arises from a more detailed expansion of \( T_2 \) when considering the interaction between the errors in the function and its derivatives. Combining these estimates yields the desired result.
	\end{proof}
	
	This result formally confirms that the curvature of the CFIF depends in a Lipschitz-continuous manner on the vertical scaling parameters. Consequently, the proposed method is stable, as small perturbations in the scaling factors, whether from optimization, numerical computation, or noise, induce only proportionally small changes in the resulting geometry.
	
	\section{Numerical Results}\label{sec5}
	
	We present numerical experiments for three data sets exhibiting varying degrees of geometric complexity. In each case, we compare the curvature of the classical cubic spline interpolant $S(x)$ with that of the proposed fractal interpolant $F(x)$. The curvature error, defined as $|\kappa_F(x) - \kappa_S(x)|$, and the root mean square error (RMSE) between $F(x)$ and $S(x)$ are computed and summarized with respect to the appropriate scaling factors $\lambda_s$ obtained from the Theorem \ref{theorem4}:
	
	\subsection{Test Case 1: Low-Curvature Data}
	
	We begin with a nearly linear data set: $ \{(0, 1),\ (1, 1.5),\ (2, 2),\ (3, 2.5),\ (4, 3)\},$ which represents a low-curvature scenario. In this case, it is straightforward to verify that the maximum slope is $K = 0.5$ and the second derivative vanishes, i.e., $M = 0$. Therefore, according to Theorem~\ref{theorem4}, the scaling factors can be chosen such that $\lambda_i < 1$. Since the data segments are piecewise linear, the curvature is identically zero, and the curvature constraint is automatically satisfied for any $\lambda_i < 1$. 	Figure~\ref{fig:curv_low_combined} shows the graphs of the classical cubic spline, the corresponding curvature-preserving cubic fractal interpolation function, and the curvature profiles of both interpolants. As anticipated, the curvature deviation is negligible in this low-complexity case, affirming the effectiveness of the proposed method.
	
	\begin{figure}[!htb]
		\centering
		\begin{tabular}{ccc}
			\subfloat[Classical Cubic Spline $S$]{\includegraphics[width=4.2cm]{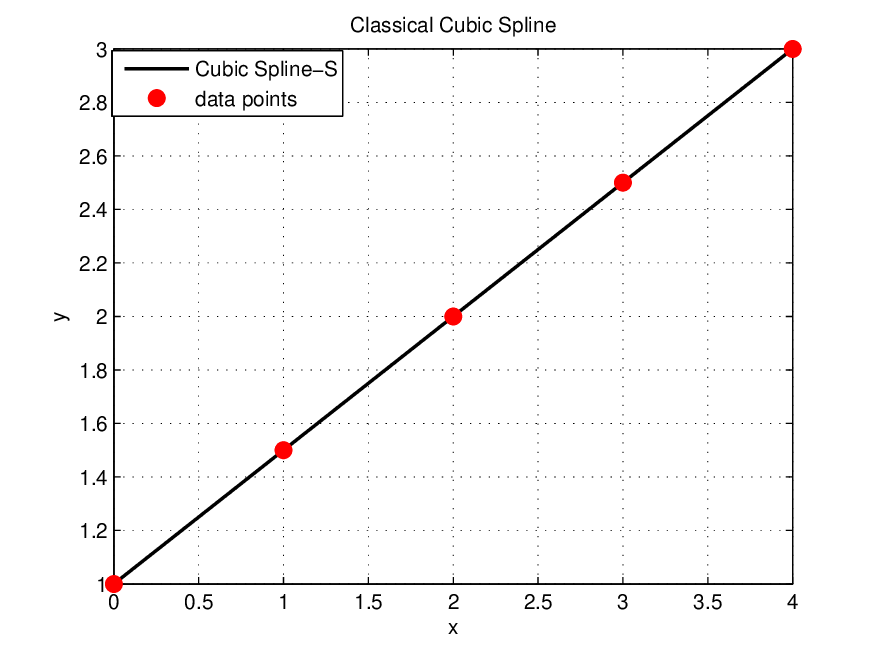}\label{fig_cl_4_d1}} &
			\subfloat[Cubic FIF $F$]{\includegraphics[width=4.2cm]{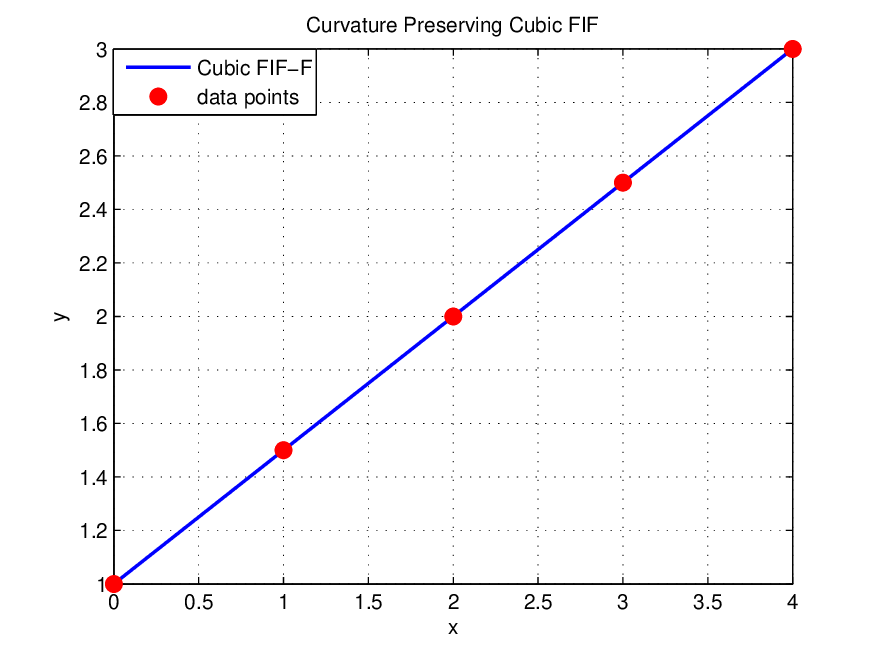}\label{fig_CFIF_4_d1}} & 
			\subfloat[Curvature Comparison]{\includegraphics[width=4.2cm]{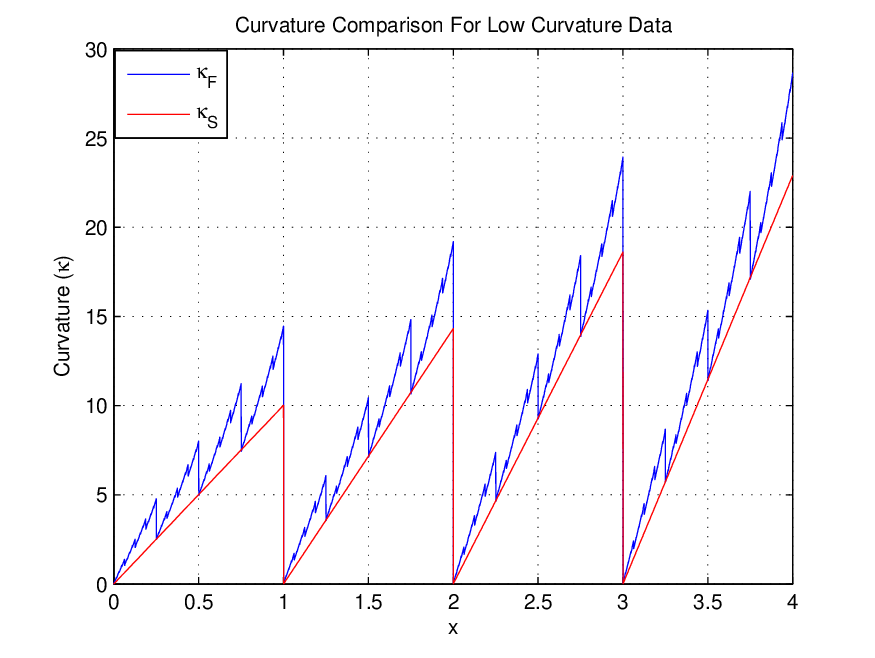}\label{curv_comp_d1}}
		\end{tabular}
		\caption{Curvature Comparison and Error for Low-Curvature Data}
		\label{fig:curv_low_combined}
	\end{figure}
	
	\subsection{Test Case 2: High-Curvature Data}
	We now examine a data set exhibiting a sharp geometric bend: $\{(0, 0),(1, 2),(2, -1),\\ (3, 2),(4, 0)\}$ constructed to evaluate curvature fidelity in the presence of rapid changes in slope. In this case, the curvature deviation is significantly higher, particularly in regions with large gradients. A straightforward analysis yields $K \approx 3.587$ as the maximum slope and $M \approx 14.571$ as the approximate maximum of the second derivative. According to Theorem~\ref{theorem4}, the scaling factors must satisfy $\lambda_i < 0.303$ to ensure the curvature deviation remains within a tolerable range. Figure~\ref{fig:curv_high_combined}  shows the classical cubic spline, the curvature-preserving fractal interpolation function, and a comparison of their curvature profiles.
	
	\begin{figure}[!htb]
		\centering
		\begin{tabular}{ccc}
			\subfloat[Classical Cubic Spline $S$]{\includegraphics[width=4.2cm]{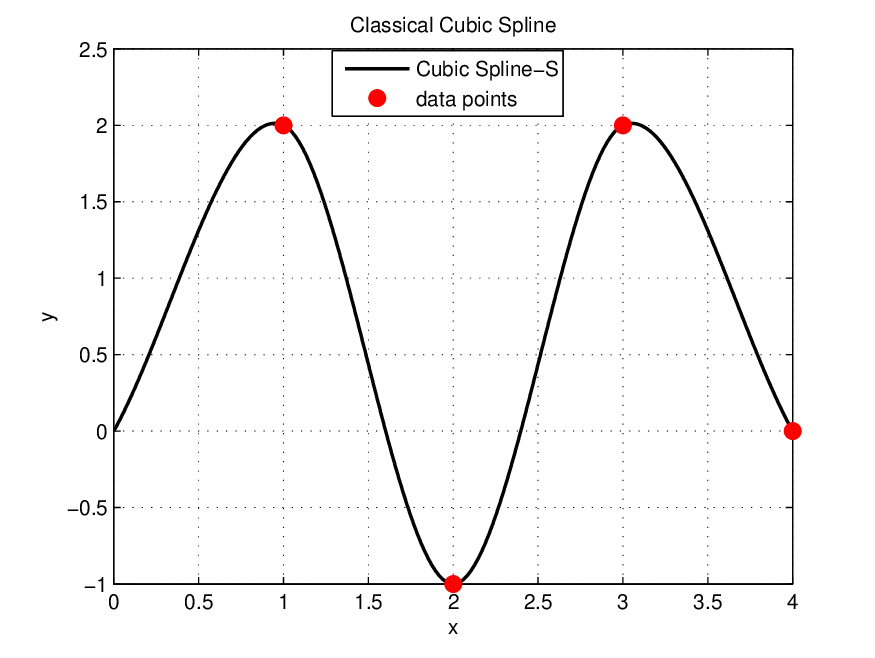}\label{fig_cl_4_d2}} &
			\subfloat[Cubic FIF $F$]{\includegraphics[width=4.2cm]{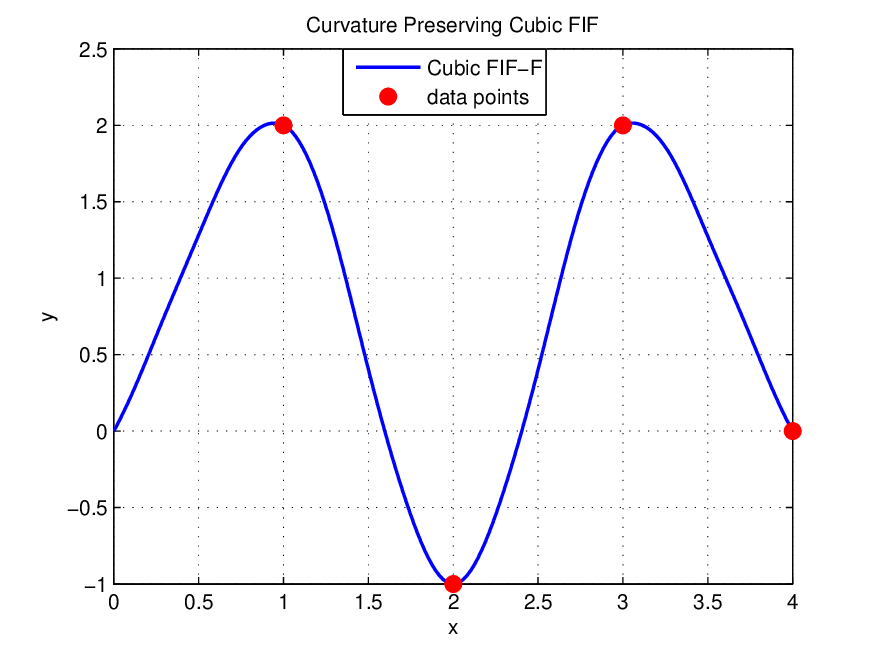}\label{fig_CFIF_4_d2}} & 
			\subfloat[Curvature Comparison]{\includegraphics[width=4.2cm]{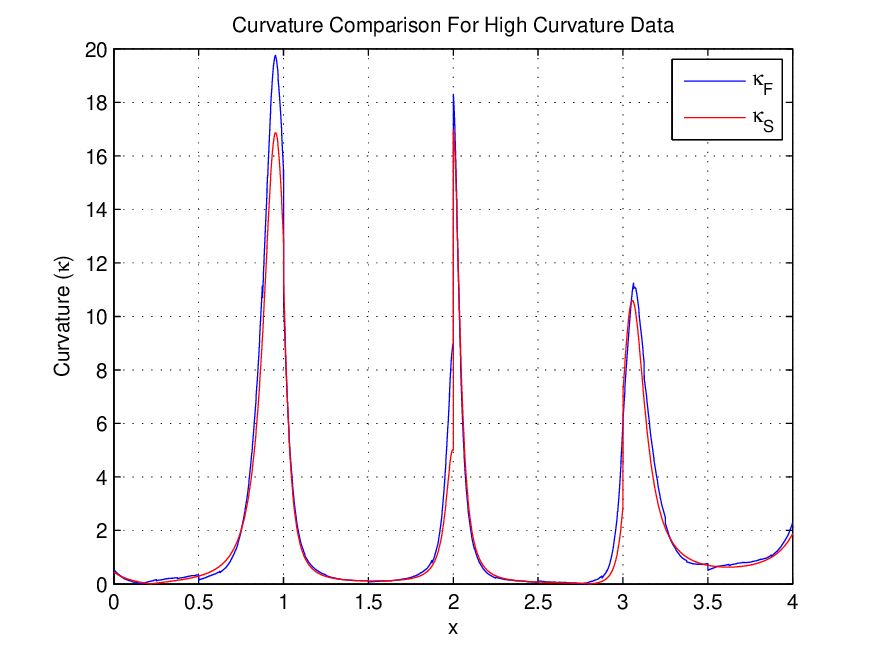}\label{curv_comp_d2}}
		\end{tabular}
		\caption{Curvature Comparison and Error for High-Curvature Data}
		\label{fig:curv_high_combined}
	\end{figure}
	
	\subsection{Test Case 3: Noisy Data}
	Finally, we consider a noisy data set generated by perturbing a sine function: $y_i = \sin(x_i) + \varepsilon_i, \ \varepsilon_i \sim \mathcal{N}(0, 0.1),$ sampled over a uniform grid. We took $\epsilon_i = 0.1$. This test case mimics real-world scenarios in which otherwise smooth data is corrupted by random measurement noise. It is readily observed that the maximum slope is approximately $K \approx 1.310$ and the estimated maximum second derivative is $M \approx 2.559$. According to Theorem~\ref{theorem4}, the scaling factors should satisfy the constraint $\lambda_i < 0.149$ to ensure the curvature error remains bounded. Figure~\ref{fig:curv_noisy_combined} displays the classical cubic spline, the curvature-preserving fractal interpolation function, and a comparison of their curvature profiles.
	
	\begin{figure}[!htb]
		\centering
		\begin{tabular}{ccc}
			\subfloat[Classical Cubic Spline $S$]{\includegraphics[width=4.2cm]{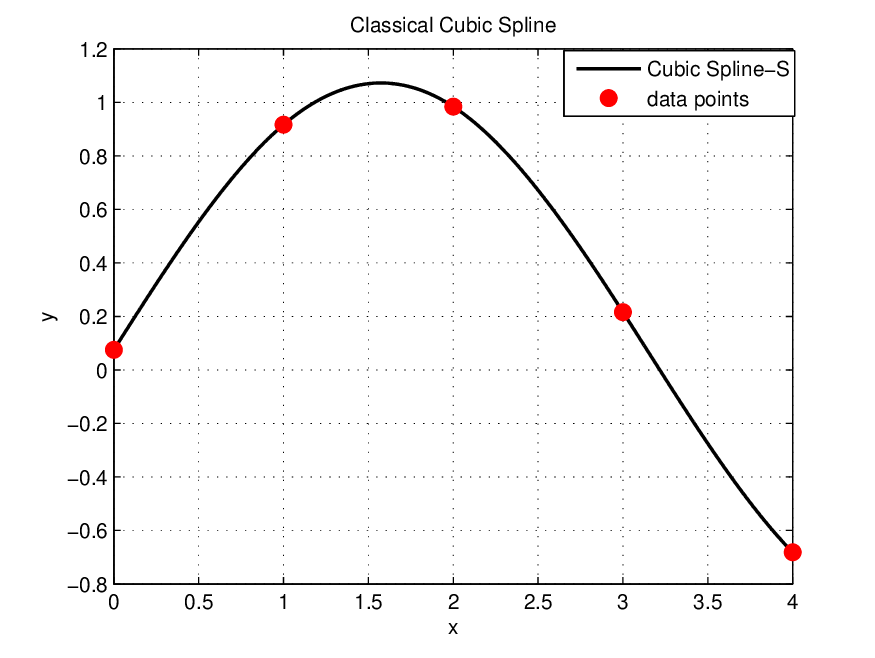}\label{fig_cl_4_d3}} &
			\subfloat[Cubic FIF $F$]{\includegraphics[width=4.2cm]{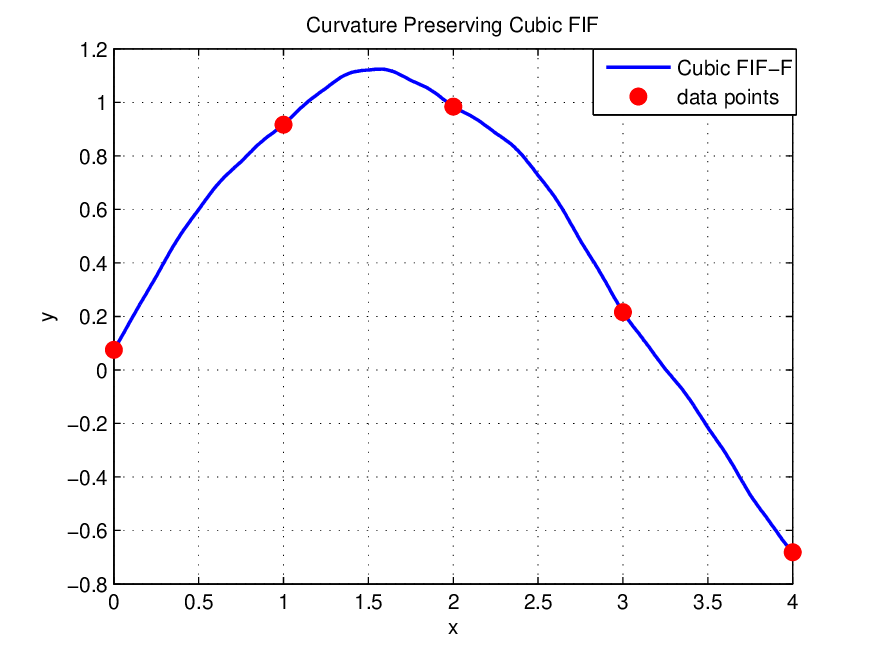}\label{fig_CFIF_4_d3}} &
			\subfloat[Curvature Comparison]{\includegraphics[width=4.2cm]{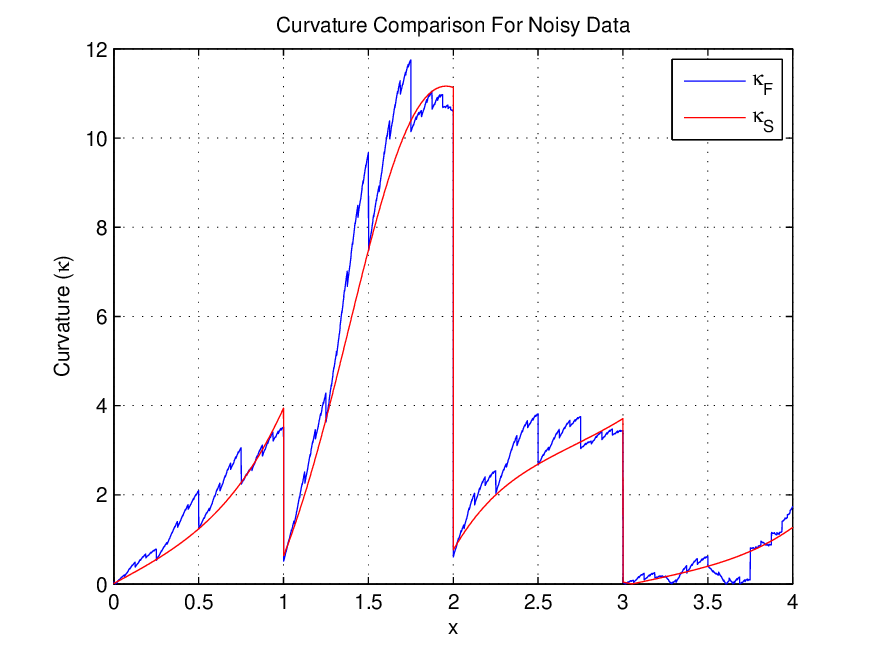}\label{curv_comp_d3}}
		\end{tabular}
		\caption{Curvature Comparison and Error for Noisy Data}
		\label{fig:curv_noisy_combined}
	\end{figure}
	
	\noindent
	These results in Table~\ref{tab:error} demonstrate that the curvature-preserving fractal interpolant performs well in low-curvature regimes and achieves better fidelity than classical cubic splines in the presence of noise or sharp transitions.
	
	We conducted timing experiments comparing our curvature-preserving cubic FIF with classical interpolation methods (linear, cubic spline, and PCHIP) across all datasets. Table~\ref{tab:timing} shows the average execution times for different evaluation point counts. The results indicate a computational overhead of approximately 200--250$\times$ compared to cubic splines, consistent across datasets. This cost stems from the iterative fractal refinement and curvature optimization. However, absolute execution times remain practical (sub-second), and the trade-off is justified by the superior curvature preservation shown earlier. Computational complexity exhibits approximately linear scaling ($O(n^{0.16})$ for low-curvature data), while classical methods show near-constant time complexity. Our cubic FIF framework provides a fundamental approach to geometric fidelity. While advanced schemes like rational fractal interpolation \cite{Nav2002}-\cite{VNC2016} offer additional shape control, they increase parameter complexity and computational cost. Our approach balances mathematical tractability with curvature preservation guarantees. Future work includes extensions to rational fractal schemes and higher-dimensional modeling, which would further illuminate trade-offs between model complexity and geometric fidelity.
	
	\begin{table}[!htb]
		\centering
		\caption{RMSE and Maximum Curvature Error for Various Datasets}
		\label{tab:error}
		\begin{tabular}{|c|c|c|c|} \hline
			Dataset Type & $\lambda$ Values & RMSE $\|F - S\|$ & Max Curvature Error \\ \hline
			Low Curvature & $[0.01,\ 0.011,\ 0.012,\ 0.013]$ & $4.4408 \times 10^{-16}$ & $5.7674$ \\ 
			High Curvature & $[0.01,\ 0.011,\ 0.012,\ 0.013]$ & $0.0390$ & $3.9932$ \\ 
			Noisy Data & $[0.01,\ 0.011,\ 0.012,\ 0.013]$ & $0.0060$ & $2.1922$ \\ \hline
		\end{tabular}
	\end{table}
	
	\begin{figure}[!htb]
		\centering
		\begin{tabular}{ccc}
			\subfloat[First Derivative of CFIF: Lower-Curvature Data]{\includegraphics[width=4.2cm, height = 3.5 cm]{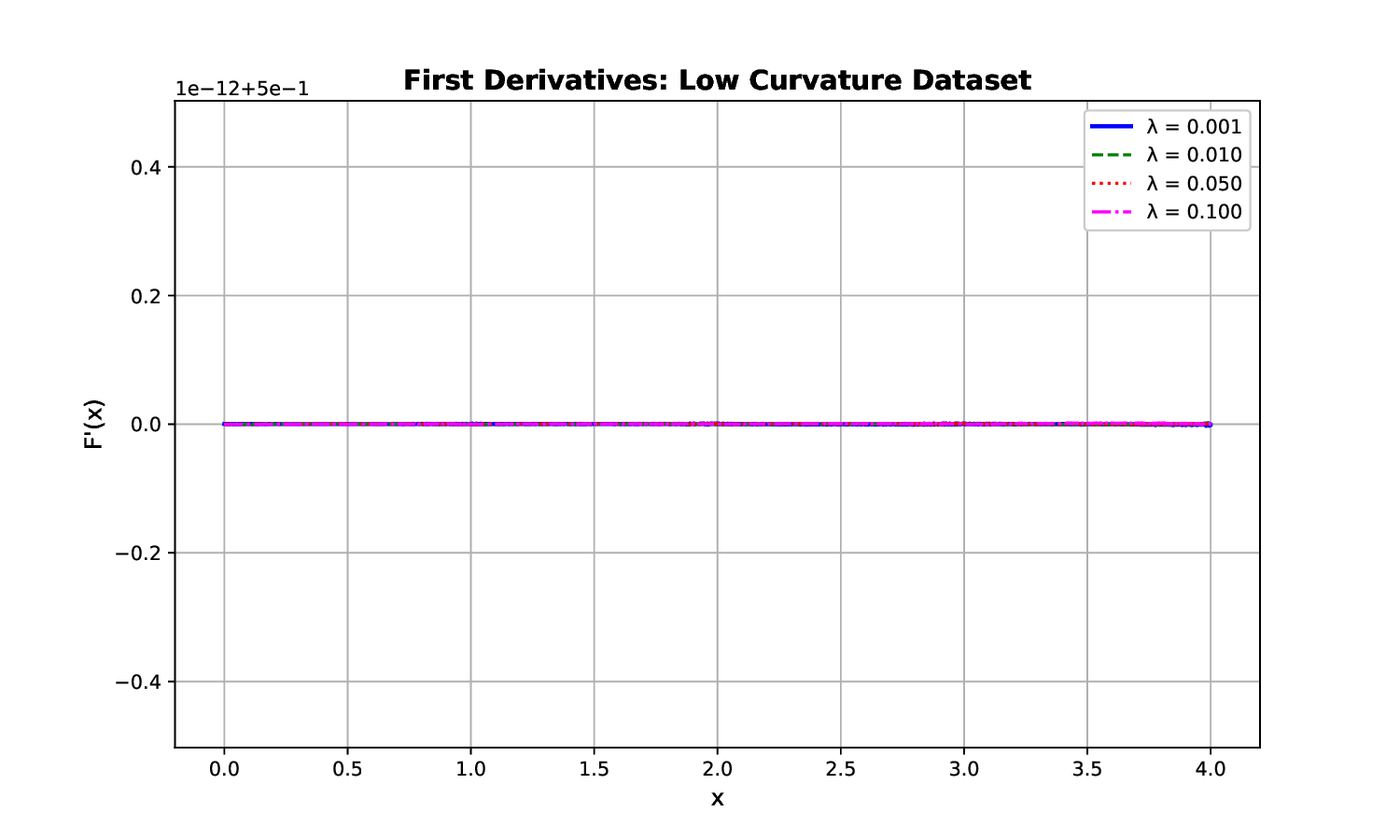}\label{1st_der_CFIF_LCD}} &
			\subfloat[First Derivative of CFIF: Lower-Curvature Data]{\includegraphics[width=4.2cm, height = 3.5 cm]{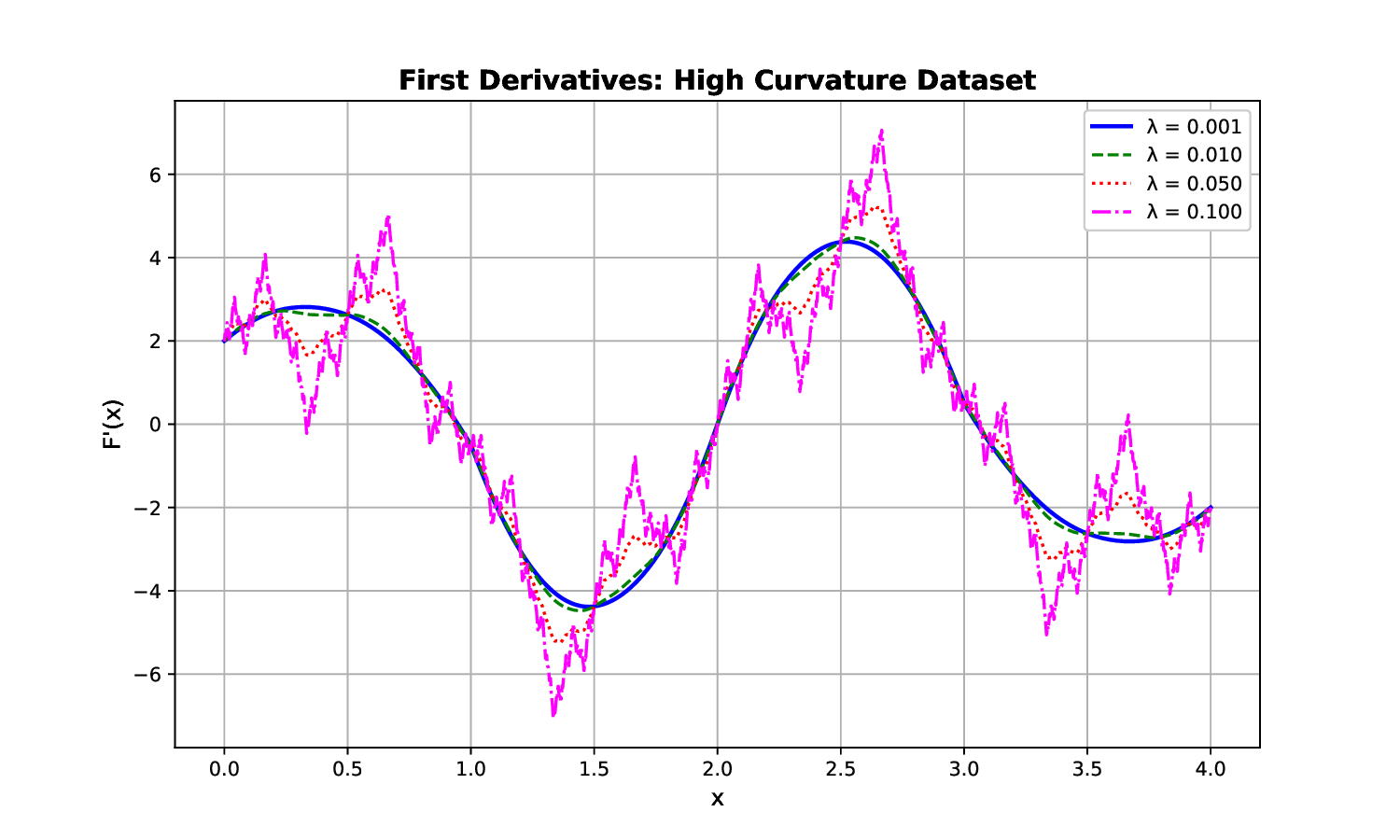}\label{1st_der_CFIF_HCD}} &
			\subfloat[First Derivative of CFIF: Lower-Curvature Data]{\includegraphics[width=4.2cm, height = 3.5 cm]{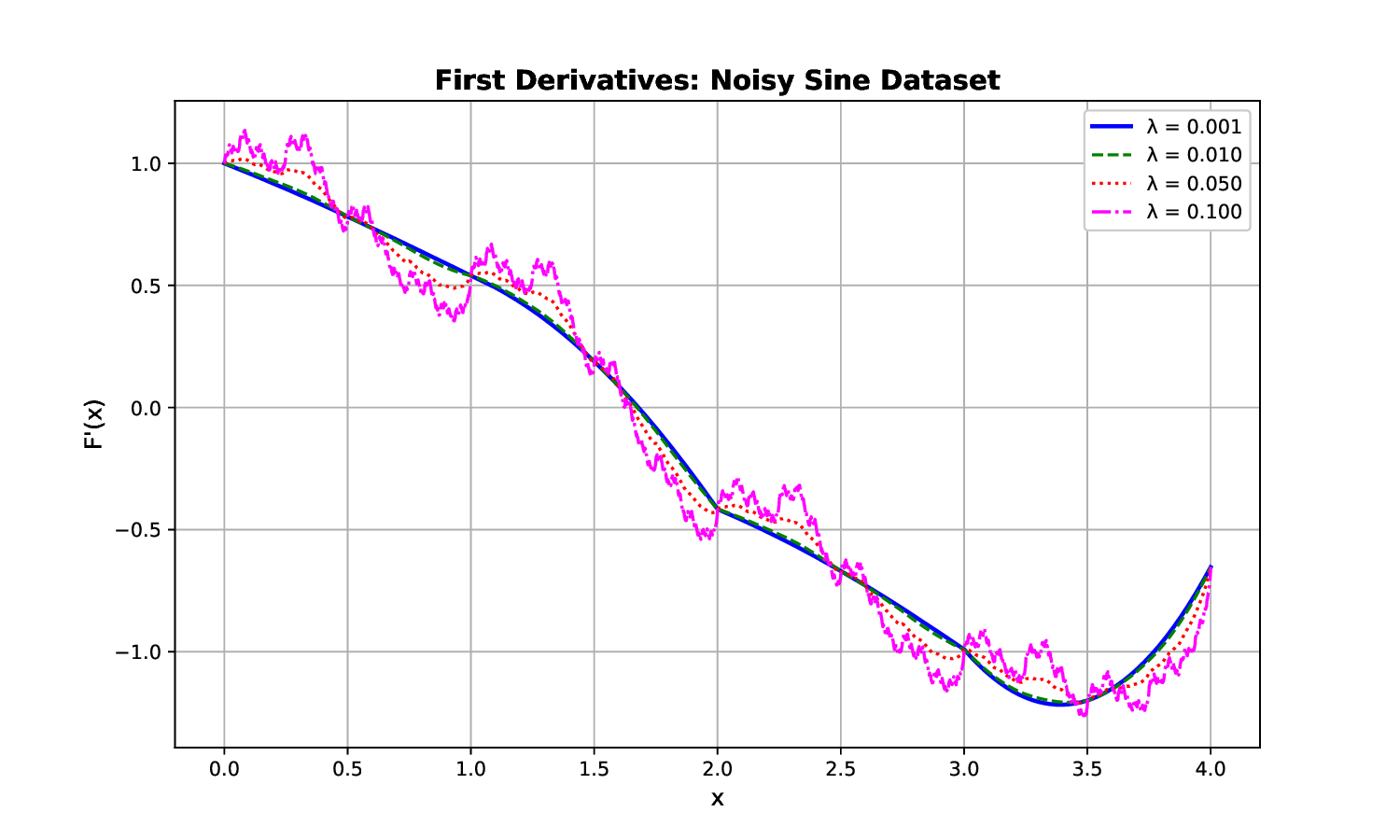}\label{1st_der_CFIF_NS}}
		\end{tabular}
		\caption{Parameter sensitivity analysis showing first derivatives for different scaling factors $\lambda$ across all datasets.}
		\label{fig:param_sensitivity}
	\end{figure}
	
	\begin{table}[htbp]
		\centering
		\caption{Computational performance comparison (execution times in seconds)}
		\label{tab:timing}
		\begin{tabular}{|l|c|c|c|c|c|c|c|}
			\hline
			\textbf{Dataset} & \textbf{Method} & \textbf{10 pts} & \textbf{20 pts} & \textbf{50 pts} & \textbf{100 pts} & \textbf{200 pts} & \textbf{500 pts} \\
			\hline
			Low & Linear & 0.00021 & 0.00011 & 0.00012 & 0.00010 & 0.00011 & 0.00013 \\
			Curvature & Cubic & 0.00078 & 0.00024 & 0.00025 & 0.00023 & 0.00023 & 0.00023 \\
			& PCHIP & 0.00028 & 0.00026 & 0.00026 & 0.00025 & 0.00025 & 0.00026 \\
			& \textbf{FIF} & \textbf{0.05463} & \textbf{0.05074} & \textbf{0.05117} & \textbf{0.05470} & \textbf{0.08347} & \textbf{0.09417} \\
			\hline
			High & Linear & 0.00014 & 0.00010 & 0.00009 & 0.00010 & 0.00015 & 0.00011 \\
			Curvature & Cubic & 0.00026 & 0.00027 & 0.00023 & 0.00029 & 0.00023 & 0.00030 \\
			& PCHIP & 0.00033 & 0.00140 & 0.00027 & 0.00031 & 0.00026 & 0.00058 \\
			& \textbf{FIF} & \textbf{0.05848} & \textbf{0.05235} & \textbf{0.05086} & \textbf{0.05411} & \textbf{0.05369} & \textbf{0.05389} \\
			\hline
			Noisy & Linear & 0.00012 & 0.00009 & 0.00009 & 0.00009 & 0.00010 & 0.00010 \\
			Sine & Cubic & 0.00025 & 0.00022 & 0.00022 & 0.00022 & 0.00022 & 0.00023 \\
			& PCHIP & 0.00028 & 0.00061 & 0.00026 & 0.00025 & 0.00025 & 0.00026 \\
			& \textbf{FIF} & \textbf{0.05410} & \textbf{0.05054} & \textbf{0.05100} & \textbf{0.05270} & \textbf{0.05324} & \textbf{0.05343} \\
			\hline
		\end{tabular}
	\end{table}
	
	We analyzed the effect of vertical scaling factors $\lambda_i$ on our curvature-preserving cubic FIF. Figure~\ref{fig:param_sensitivity} shows first derivatives for $\lambda \in [0.001, 0.1]$ across all datasets. Results demonstrate predictable sensitivity: smaller $\lambda$ values yield near-classical behavior, while larger values introduce controlled fractal oscillations. The method maintains stability across $\lambda \in [0.001, 0.1]$, with optimal results in $[0.01, 0.05]$ for most applications. This controlled sensitivity ensures practical usability without extensive parameter tuning.
	
	\section{Conclusion}
	We introduced a curvature-preserving fractal interpolation framework that merges cubic splines with fractal perturbations. Our method constructs interpolants that pass through data points while closely matching the underlying spline's curvature profile. A penalty-based optimization minimized curvature deviation, ensuring high geometric fidelity for both smooth and highly curved datasets, even under noise. Numerical experiments confirm our method's superior performance in preserving curvature and interpolation accuracy. Future work will extend this to rational or Hermite-type fractal interpolants as well as generalization to surface interpolation and 3D modeling, where curvature preservation becomes particularly important for applications in computer-aided design, biomedical modeling, and geometry-aware data fitting, etc.


\begin{thebibliography}{9}
		\bibitem{FT91} Fiorot, J. C., and Tabka, J.: Shape-preserving $\mathcal{C}^2$ cubic polynomial interpolating splines. Mathematics of Computation \textbf{57}195, 291–-298 (1991)
		\bibitem{HSS11} Hussain, M. Z., Sarfraz, M., and Shaikh, T. S.: Shape preserving rational cubic spline for positive and convex data. Egyptian Informatics Journal \textbf{12}(3), 231–-236 (2011)
		\bibitem{LM01} Lamberti, P., and Manni, C.: Shape-preserving $\mathcal{C}^2$ functional interpolation via parametric cubics. Numerical Algorithms \textbf{28}, 229–-254 (2001)
		\bibitem{SHS11} Sarfraz, M., Hussain, M. Z., Shaikh, T. S., and Iqbal, R.: Data visualization using shape preserving $\mathcal{C}^2$ rational spline. Proceedings of the 15th International Conference on Information
		Visualisation (IV-11), pp 528–-533, London, (2011)
		\bibitem{DWT05} Duan, Q., Wang, L., and Twizell, E. H.: A new $\mathcal{C}^2$ rational interpolation based on function values and constrained control of the interpolant curves. Applied Mathematics and Computation \textbf{161}(1), 311–-322 (2005)
		\bibitem{SHH12} Sarfraz, M., Hussain, M. Z., and Hussain, M.: Shape-preserving curve interpolation. International Journal of Computer Mathematics \textbf{89}(1), 35–-53 (2012)
		\bibitem{HG1991} Greiner, H.: A survey on univariate data interpolation and approximation by splines of given shape. Math. Comput. Model. \textbf{15}(10), 97–-108 (1991)
		\bibitem{Barnsley86} Barnsley, M. F.: Fractal Functions and Interpolation. Constructive Approximation \textbf{2}(4), 303–-329 (1986)
		\bibitem{Barnsley93} Barnsley, M. F.: Fractals Everywhere. 2nd Sub edn. Morgan Kaufmann Pub. (1993)
		\bibitem{BarHar89} Barnsley, M. F., Harrington, A. N.: The calculus of fractal interpolation functions. J. Approx. Theory. \textbf{57}(1), 14--34 (1989)
		\bibitem{CN2009} Chand, A. K. B., and Navascu\'{e}s, M. A.: Generalized Hermite Fractal Interpolation, Rev Real Academia de Ciencias, Zaragoza. \textbf{64}, 107–-120 (2009)
	\bibitem{CVN2014} Jha, S., Chand, A. K. B., and Navascu\'{e}s, M. A.: Approximation by shape preserving fractal functions with variable scalings. Calcolo \textbf{58}(1), (2021)
		\bibitem{CV2015} Chand, A. K., Viswanathan, P.: Rational Cubic Spline FIFs Preserving Shape. Applied Mathematics and Computation. \textbf{251}, 278–-292 (2015)
		\bibitem{Nav2002} Navascu\'{e}s, M. A.: Fractal Approximation. Revista de la Real Academia de Ciencias \textbf{96}(1), 73–-84 (2002)
		\bibitem{TCS2021} Tyada, K. R., Chand, A. K. B., Sajid, M.: Shape Preserving Rational Cubic Trigonometric Fractal Interpolation Functions. Mathematics and Computers in Simulation \textbf{190}, 866--891 (2021)
	\bibitem{VC2014} Navascu\'{e}s, M. A., Chand, A. K. B., Veedu, V. P., and Sebasti\'{a}n, M. V.: Fractal interpolation functions: a short survey. Applied Mathematics \textbf{5}(12), 1834--1841 (2014)
	\bibitem{VNC2016} Viswanathan, P., Chand, A. K. B., and Navascu\'{e}s, M. A.: Fractal perturbation preserving fundamental shapes: Bounds on the scale factors. Journal of Mathematical Analysis and Applications \textbf{419}(2), 804–-817 (2014)
		\bibitem{CK2015} Chand, A. K. B., Tyada, K. R.: Constrained 2D data interpolation using rational cubic fractal functions. In proceedings of International conference on Recent Trends in Mathematical Analysis and its Applications-ICRTMAA-14, pp 593--607. Springer, India (2015).
		\bibitem{chand2015positivity} Chand, A. K. B., and Tyada, K. R.: Positivity preserving rational cubic trigonometric fractal interpolation functions. In proceedings of the 1st International Conference on Mathematics and Computing: ICMC-15, pp 187--202. Sringer India (2015).
		\bibitem{CK2018} Chand, A. K. B., Tyada, K. R.: Constrained Shape Preserving Rational Cubic Fractal Interpolation Functions. Rocky Mt. J. Math. \textbf{48}(1), 75–-105 (2018)
		\bibitem{KC2018} Chand, A. K. B., Tyada, K. R.: Shape Preserving Constrained and Monotonic Rational Quintic Fractal Interpolation Functions. Int. J. Adv. Eng. Sci. Appl. Math. \textbf{10}(1), 15–-33 (2018)
		
	\end{thebibliography}
\end{document}